\documentclass{amsart}
\usepackage{amsmath}
\usepackage{amssymb}
\usepackage{amsthm}

\newtheorem{theorem}{Theorem}

\newtheorem{lemma}{Lemma}
\newtheorem{rem}{Remark}

\newtheorem{cor}{Corollary}

\begin{document}

\title{The graphs of projective codes}
\author{Mariusz Kwiatkowski, Mark Pankov, Antonio Pasini}
\address{Faculty of Mathematics and Computer Science, 
University of Warmia and Mazury, S{\l}oneczna 54, Olsztyn, Poland}
\email{mkw@matman.uwm.edu.pl, pankov@matman.uwm.edu.pl}
\address{Department of information engineering and mathematics, University of Siena, Via Roma 56, Siena, Italy}
\email{pasini@unisi.it}

\maketitle

\begin{abstract}
Consider the Grassmann graph formed by $k$-dimensional subspaces of an $n$-dimensional vector space 
over the field of $q$ elements ($1<k<n-1$) and 
denote by $\Pi(n,k)_q$ the restriction of this graph to the set of projective $[n,k]_q$ codes.
In the case when $q\ge \binom{n}{2}$, we show that
the graph $\Pi(n,k)_q$ is connected, its diameter is equal to the diameter of the Grassmann graph
and the distance between any two vertices coincides with the distance between these vertices in the Grassmann graph.
Also, we give some observations concerning the graphs of simplex codes. 
For example, binary simplex codes of dimension $3$ are precisely maximal singular subspaces of 
a non-degenerate quadratic form.
\end{abstract}

\section{Introduction}
The Grassmann graph formed by $k$-dimensional subspaces of an $n$-dimensional vector space over 
the field of $q$ elements (we always suppose that $1<k<n-1$) can be considered as the graph of all linear $[n,k]_q$ codes, where
two distinct codes are adjacent vertices if  they have the maximal possible number of common codewords.
This is one of the classical examples of distance regular graphs \cite[Section 9.3]{drg-book}. 
Recall that the distance between two vertices in a connected graph $\Gamma$ 
is the number of edges in a shortest path connecting these vertices.
If $\Gamma'$ is a connected subgraph of $\Gamma$, then the distance between two vertices of $\Gamma'$
can be greater than the distance between these vertices in $\Gamma$.

It is natural to reduce  the study of linear codes to non-degenerate codes only, 
i.e. the case when generator matrices do not contain zero columns.
The restriction of the Grassmann graph to the set of non-degenerated linear $[n,k]_q$ codes is investigated in \cite{KP1,KP2}.
This graph is connected, but the distance between any two vertices coincides with 
the distance between these vertices in the Grassmann graph only in the case when
$$n<(q+1)^2 +k-2.$$
In this paper, we consider the subgraph of the Grassmann graph formed by all projective codes, 
i.e. linear codes whose generator matrices do not contain proportional columns.
Projective $[n,k]_q$ codes can be obtained from $n$-element subsets of $(k-1)$-dimensional projective spaces over the field of $q$ elements 
\cite[Section 1.1]{TVN} (see, for example, \cite{CG,N,R} for projective codes associated to Grassmannians embedded in projective spaces).
Our main result (Theorem 1) is the following: if $q\ge \binom{n}{2}$,
then the graph of projective $[n,k]_q$ codes is connected, its diameter is equal to the diameter of the Grassmann graph
and the distance between any two vertices coincides with the distance between these vertices in the Grassmann graph.

Simplex codes (linear codes dual to Hamming codes) are projective codes which do not satisfy the above inequality,
but the latter statement holds for the graph of $3$-dimensional binary simplex codes.
We show that simplex codes can be described by a system of polynomial equations
(the number of equations and their degrees depend  on the code dimension and the size of the field).
In particular, binary simplex codes of dimension $3$ are precisely maximal singular subspaces of 
a non-degenerate quadratic form on a $6$-dimensional vector space over the field of two elements.
The latter implies that the graph of binary simplex codes of dimension $3$ is isomorphic to 
the incidence graph formed by $1$-dimensional and $3$-dimensional subspaces of a $4$-dimensional vector space over 
the two-element field.
We  also consider examples related to the graphs of $4$-dimensional binary simplex codes and $3$-dimensional ternary simplex codes.

\section{Basic objects}
\subsection{Grassmann graphs}
Let $\mathbb{F}_{q}$ be the finite field consisting of $q$ elements.
Consider the $n$-dimensional vector space 
$$V=\underbrace{{\mathbb F}_{q}\times\dots \times{\mathbb F}_{q}}_{n}$$
over this field.
The Grassmann graph $\Gamma_{k}(V)$ is the graph whose vertices are $k$-dimensional subspaces of $V$
and two such subspaces are adjacent vertices of this graph if their intersection is $(k-1)$-dimensional.
The graph $\Gamma_{k}(V)$ is connected and 
the distance between any two $k$-dimensional subspaces  $X$ and $Y$ in this graph is equal to 
$$k-\dim(X\cap Y).$$
The diameter of $\Gamma_{k}(V)$ is equal to $k$ if $n\ge 2k$ and the diameter is $n-k$ if $n<2k$.

Recall that there are precisely 
$$\frac{q^{m}-1}{q-1}=q^{m-1}+\dots+q+1$$
distinct $1$-dimensional subspaces in an $m$-dimensional vector space over $\mathbb{F}_{q}$.
In what follows, this number will be denoted by $[m]_q$.
The number of hyperplanes in this vector space also is equal to $[m]_q$.

Suppose that the distance between $k$-dimensional subspaces $X$ and $Y$ in the Grassmann graph is equal to $m$. 
Then $X\cap Y$ is $(k-m)$-dimensional. 
It is well-known that every geodesic in $\Gamma_{k}(V)$ connecting $X$ and $Y$ is formed by 
$k$-dimensional subspaces containing $X\cap Y$.
The number of hyperplanes of $X$ containing $X\cap Y$ and  
the number of $(k-m+1)$-dimensional subspaces of $Y$ containing $X\cap Y$ both are equal to $[m]_q$.
This implies that there are precisely $[m]_q^2$ vertices of $\Gamma_{k}(V)$ adjacent to $X$ and at distance $m-1$ from $Y$.
Applying the same arguments to distances less than $m$, we establish that 
the number of geodesics in $\Gamma_{k}(V)$ connecting $X$ and $Y$ is equal to
$$[m]_{q}^{2}[m-1]_{q}^2\dots [2]_{q}^2.$$
See \cite[Section 9.3]{drg-book} for more properties of Grassmann graphs.
 
\subsection{Linear codes}
A {\it linear} $[n,k]_{q}$ {\it code} $C$ is a $k$-dimensional subspace of $V$;
non-zero vectors belonging to this subspace are called {\it codewords} of $C$.

The standard basis of $V$ is formed by the vectors 
$$e_{1}=(1,0,\dots,0),\dots,e_{n}=(0,\dots,0,1).$$
Denote by $c_{i}$ the $i$-th coordinate functional $(x_{1},\dots,x_{n})\to x_{i}$
and write $C_{i}$ for the hyperplane of $V$ which is the kernel of $c_i$.

A linear $[n,k]_{q}$ code $C$ is {\it non-degenerate} if the restriction of every coordinate functional to $C$ is non-zero.
In this case, we write $P_{i}$ for the $1$-dimensional subspace of $C^{*}$ containing $c_{i}|_{C}$.
Note that the equality $P_{i}=P_{j}$ is possible for some distinct $i,j$.
The collection ${\mathcal P}(C)$ formed by all $P_{i}$ is called the {\it projective system} associated to the linear code $C$ \cite[Section 1.1]{TVN}.
Since the code $C$ is assumed to be non-degenerated, the projective system contains $k$ elements whose sum coincides with $C^{*}$.
Our linear code is said to be {\it projective} if all $P_{i}$ are mutually distinct, i.e. 
${\mathcal P}(C)$ is an $n$-element subset of the projective space corresponding to $C^{*}$.

Let $x_{1},\dots,x_{k}$ be a basis of $C$ and let $x^{*}_{1},\dots,x^{*}_{k}$ be the dual basis of $C^{*}$, 
i.e. $x^{*}_{i}(x_{j})=\delta_{ij}$ (Kronecker delta).
Consider the generator matrix $M$ of $C$ whose rows are $x_{1},\dots,x_{k}$.
If $(a_{1j},\dots, a_{kj})$ is the $j$-column of $M$, then 
$$c_{j}|_{C}= a_{1j}x^{*}_{1}+\dots+a_{kj}x^{*}_{k}.$$
Therefore, the code $C$ is projective if and only if any two columns in $M$ are non-proportional.
The latter condition is equivalent to the fact that $C\cap C_{i}$ and $C\cap C_{j}$ are distinct hyperplanes of $C$ for any distinct $i,j$.

For fixed $n$ and $q$ the following assertions are fulfilled:
\begin{enumerate}
\item[$\bullet$] for a given number $k\in \{2,\dots,n-2\}$  projective $[n,k]_q$ codes exist if and only if $[k]_{q}\ge n$;
\item[$\bullet$] projective $[n,k]_q$ codes exist for all $k\in \{2,\dots,n-1\}$ if and only if $q+1\ge n$.
\end{enumerate}
In the case when $n=[k]_{q}$, projective $[n,k]_q$ codes are known as {\it $q$-ary simplex codes of dimension $k$}
(see, for example, \cite[Section 1.8]{HP-book}).
The projective system of such a code consists of all points of a $(k-1)$-dimensional projective space.
By general properties of projective systems (see \cite[Theorem 1.1.6]{TVN} or \cite[Lemma 2.20]{Pank-book}),
this implies that any two $q$-ary simplex codes of dimension $k$ are equivalent, i.e.
there is a monomial linear automorphism of $V$ transferring one of these codes to the other.

\section{The distance between projective codes}

\subsection{Results}
Denote by $\Pi(n,k)_{q}$ the restriction of the Grassmann graph $\Gamma_{k}(V)$ to the set of all projective $[n,k]_{q}$ codes,
i.e. the vertices of the graph $\Pi(n,k)_{q}$ are projective $[n,k]_{q}$ codes and two such codes are adjacent vertices of the graph
if their intersection is $(k-1)$-dimensional.

In the case when $k=1,n-1$, any two distinct vertices of $\Gamma_{k}(V)$ are adjacent.
For this reason, we will alway suppose that $1<k<n-1$.

\begin{theorem}
Suppose that $q\ge \binom{n}{2}$.
Then $\Pi(n,k)_{q}$ is a connected graph of diameter $$\min\{k,n-k\},$$
the distance between any two projective $[n,k]_{q}$ codes $X$ and $Y$ in this graph is equal to $$k-\dim(X\cap Y)$$
{\rm(}i.e. it coincides with the distance between $X$ and $Y$ in the Grassmann graph{\rm)}.
If this distance is equal to $m$, then there are at least 
$$[m]_{q}[m-1]_{q}\dots[2]_{q}$$
distinct geodesics in $\Pi(n,k)_q$ connecting $X$ and $Y$.
\end{theorem}
The condition $q\ge \binom{n}{2}$ does not hold if $n=[k]_q$ (the case of simplex codes).
In the case when
$$n=15,\;k=4,\;q=2\;\mbox{ or }\;n=13,\;k=3,\;q=3,$$
there are projective $[n,k]_q$ codes $X,Y$ such that $\dim (X\cap Y)=k-2$
and there is no projective $[n,k]_q$ code adjacent to both $X$ and $Y$ (Subsection 4.2).

On the other hand, the graph $\Pi(7,3)_2$ 
is isomorphic to the incidence graph $\Gamma_{1,3}({\mathbb F}^{4}_{2})$
whose vertices are $1$-dimensional and $3$-dimensional subspaces of ${\mathbb F}^{4}_{2}$
and two distinct subspaces are adjacent vertices of the graph if they are incident.
So, $\Pi(7,3)_2$ has properties similar to the properties described in Theorem 1.

\subsection{Proof of Theorem 1}
Recall that $C_{i}$ is the kernel of the $i$-th coordinate functional.
For distinct $i$ and $j$ we denote by $C_{ij}$ the intersection of $C_{i}$ and $C_{j}$;
this is an $(n-2)$-dimensional subspace.
For every non-degenerate linear $[n,k]_{q}$ code $X$ the intersection $X\cap C_{i}$ is $(k-1)$-dimensional. 
This code is projective if and only if 
$$X\cap C_{i}\ne X\cap C_{j}$$
for any distinct $i,j$.
The equality 
$$X\cap C_{i}=X\cap C_{j}$$
is equivalent to the fact that $X\cap C_{ij}$ is $(k-1)$-dimensional. 
Therefore, a linear $[n,k]_{q}$ code $X$ is projective if and only if
all $X\cap C_{ij}$ are $(k-2)$-dimensional.

\begin{lemma}\label{lemma1-1}
Let $X$ and $Y$ be projective $[n,k]_q$ codes such that 
$$dim(X\cap Y)=k-2.$$
If $q\ge \binom{n}{2}$, then there are $q+1$ distinct projective $[n,k]_q$ codes adjacent to both $X$ and $Y$.
\end{lemma}

\begin{proof}
There are precisely $q+1$ hyperplanes of $Y$ containing $X\cap Y$.
We denote them by $H_{1},\dots,H_{q+1}$ and show that 
for every hyperplane $H\subset X$ containing $X\cap Y$ at least one of the linear $[n,k]_q$ codes $H+H_{t}$ is projective.

If all $H+H_{t}$ are non-projective codes, then for every $t\in \{1,\dots,q+1\}$ there is at least one pair $i(t),j(t)$ such that 
$$(H+H_{t})\cap C_{i(t)j(t)}$$
is a $(k-1)$-dimensional subspace. 
Since $q+1>\binom{n}{2}$ and there are precisely $\binom{n}{2}$ distinct $C_{ij}$, 
we can find distinct $t,s\in  \{1,\dots,q+1\}$ and distinct $i,j\in \{1,\dots,n\}$ such that 
$$S_{t}=(H+H_{t})\cap C_{ij}\;\mbox{ and }\;S_{s}=(H+H_{s})\cap C_{ij}$$
both are $(k-1)$-dimensional subspaces.
Suppose that $S_{t}$ coincides with $S_{s}$.
Since $H+H_{t}$ and $H+H_{s}$ are distinct and their intersection is $H$, 
the subspace $S_{t}=S_{s}$ coincides with $H$ which means that $H$ is contained in $C_{ij}$
and the code $X$ is not projective. 
Therefore, $S_{t}\ne S_{s}$.

So, the dimension of $S_{t}+S_{s}$ is not less than $k$ and this subspace is contained in $C_{ij}$. 
Since $Y$ and $S_{t}+S_{s}$ both are contained in the $(k+1)$-dimensional subspace $H+Y$,
the intersection of these subspaces is not less than $k-1$. 
Then $C_{ij}$ intersects $Y$ in a $(k-1)$-dimensional subspace which contradicts the fact that the code $Y$ is projective.
\end{proof}

\begin{lemma}\label{lemma1-2}
If $q\ge \binom{n}{2}$, then for every projective $[n,k]_q$ code $X$ and for every subspace $U\subset X$
whose dimension is less than $k-2$ there is a projective $[n,k-1]_q$ code $X'$  satisfying $U\subset X'\subset X$.
\end{lemma}

\begin{proof}
Since $X$ is a projective $[n,k]_q$ code, every $X\cap C_{i}$ is a hyperplane of $X$ and each $X\cap C_{ij}$ is $(k-2)$-dimensional.
Let $H$ be a hyperplane of $X$ different from either $X\cap C_i$ and $X\cap C_j$.
Then $H$ contains $X\cap C_{ij}$ only in the case when 
$$H\cap C_{i}=X\cap C_{ij}=H\cap C_{j}$$
which is equivalent to the fact that the linear $[n,k-1]_q$ code $H$ is not projective.
Therefore, a hyperplane $H\subset X$ is a projective $[n,k-1]_q$ code if and only if 
it does not contain any $X\cap C_{ij}$.

There are at most $\binom{n}{2}$ distinct $X\cap C_{ij}$ 
(we cannot state that $X\cap C_{ij}$ and $X\cap C_{i'j'}$ are distinct for distinct pairs $i,j$ and $i',j'$)
and each of them is contained in precisely $q+1$ distinct hyperplanes of $X$. 
Thus the number of hyperplanes of $X$ containing at least one of $X\cap C_{ij}$
is not greater than $$\binom{n}{2}(q+1).$$
The subspace $U$ is contained in precisely $[m]_q$ distinct hyperplanes of $X$, where $m$ is the codimension of $U$ in $X$.
Since $m\ge 3$ (by our assumption) and $q\ge \binom{n}{2}$, 
we have
$$[m]_q=\frac{q^m-1}{q-1}\ge \frac{q^3 -1}{q-1}=q^2+q+1\ge \binom{n}{2}q+\binom{n}{2}+1>\binom{n}{2}(q+1).$$
Therefore, there is a hyperplane of $X$ which contains $U$ and does not contain any  $X\cap C_{ij}$.
This projective $[n,k-1]_q$ code is as required.
\end{proof}

Lemma \ref{lemma1-1} can be generalized as follows.

\begin{lemma}\label{lemma1-3}
Let $X$ and $Y$ be projective $[n,k]_q$ codes such that 
$$\dim(X\cap Y)=k-m\;\mbox{ and }\;m\ge 2.$$
If $q\ge \binom{n}{2}$, then there are $[m]_q$ distinct projective $[n,k]_q$ codes $Z$ adjacent to $X$ and satisfying
$$\dim(Z\cap Y)=\dim(X\cap Y)+1.$$
\end{lemma}

\begin{proof}
In the case when $m=2$, the statement coincides with Lemma \ref{lemma1-1}.
Suppose that $m\ge 3$.
Lemma \ref{lemma1-2} implies the existence of a projective $[n,k-1]_q$ code $X'\subset X$ containing $X\cap Y$.
There are precisely $[m]_q$ distinct $(k-m+1)$-dimensional subspaces of $Y$ containing $X\cap Y$.
The sum of any such subspace and $X'$ is a projective $[n,k]_q$ code satisfying the required conditions. 
\end{proof}

Let $X$ and $Y$ be as in Lemma \ref{lemma1-3} and $q\ge \binom{n}{2}$.
Using Lemma \ref{lemma1-3}, we show recursively that the graph $\Pi(n,k)_q$ contains 
$[m]_{q}[m-1]_q\dots[2]_q$ distinct paths connecting $X$ and $Y$.
Each of these paths is of length $m$.
Since the distance between $X$ and $Y$ in the Grassmann graph $\Gamma_{k}(V)$ is equal to $m$,
all these paths are geodesics in $\Pi(n,k)_q$.
To complete the proof we need to establish the existence of projective $[n,k]_{q}$ codes $X$ and $Y$ such that 
\begin{equation}\label{eq1-1}
\dim(X\cap Y)=\max\{0, 2k-n\}.
\end{equation}
This is a consequence of the following.

\begin{lemma}\label{lemma1-4}
Suppose that $q>2$ and 
\begin{equation}\label{eq1-2}
\min\{[k]_q,[n-k]_q\}\ge n, 
\end{equation}
then there exist projective  $[n,k]_{q}$ codes $X$ and $Y$ satisfying \eqref{eq1-1}.
\end{lemma}

The inequality $q\ge \binom{n}{2}$ guarantees that the conditions of Lemma \ref{lemma1-4} hold.

\begin{proof}[Proof of Lemma \ref{lemma1-4}]
Suppose $2k\le n$.  Consider the projective $[n,k]_q$ codes $X$ and $Y$ represented by the matrices 
\[M = [I, A, B], ~~~ N = [\lambda A, I, B]\]
where $I$ is the identity matrix of rank $k$, 
the rank $k$ matrix $A=[a_{ij}]$ with $a\neq 0, 1$  is as follows 
\[A ~ = ~ \left[\begin{array}{cccccccc}
1 & 0 & ... & 0 & 0 & 0 \\
1 & 1 & ... & 0 & 0 & 0\\
0 & 1 & ... & 0 & 0 &  0\\
... & ... & ... & ... & ... & ... \\
0 & 0 & ... & 1 & 1 & a \\
0 & 0 & ... & 0 & 1& 1
\end{array}\right]\]
(the only non-zero entries are $a_{ii}=a_{i+1i}=1$ and $a_{k-1k}=a$)
and $B$ is a $k\times (n-2k)$-matrix with no null column, no column proportional to a column of $[I, A]$ and no two mutually proportional columns. Needless to say that $B$ only occurs when $2k < n$. 
Such a matrix exists, since $[k]_q- 2k\ge n-2k$.
We must choose $\lambda$ in such a way that  
the matrix
\[\left[\begin{array}{c}
M\\
N
\end{array}\right] ~ = ~ \left[\begin{array}{ccc}
I & A & B \\
\lambda A & I & B
\end{array}\right]\]
has rank $2k$. This matrix has the same rank as the matrix
\[\left[\begin{array}{ccc}
I & A & B \\
O & I-\lambda A^2 & B-\lambda AB
\end{array}\right].\]
So, we only must choose $\lambda\neq 0$ such that $\lambda^{-1}$ is not an eigenvalue of the matrix $A^2$. 
However, $A$ has at most three distinct eigenvalues , namely $1$ (of multiplicity $k-2$) and possibly $1\pm \sqrt{a}$. 
The latter exist only if $a$ is a square.   
If $q$ is odd, then $\mathbb{F}_q$ contains non-square elements and we can choose any of them as $a$. 
With such a choice of $a$, $1$ is the unique eigenvalue of $A^2$ and any $\lambda \neq 1$ is as required. 
If $q$ is even, then $A^2$ has two eigenvalues, namely $1$ and $1+a$, of multiplicity $k-2$ and $2$ (respectively). 
In this case, we have $q\geq 4$ (since $q\neq 2$ by our assumption) and $\mathbb{F}_q$ contains non-zero elements different from $1$ and $(1+a)^{-1}$. We choose any of them as $\lambda$.

Now, we suppose that $n<2k$. 
Then $2(n-k)<n$ and, by \eqref{eq1-2}, we have $[n-k]_q\ge n$.
The above arguments imply the existence of projective $[n,n-k]_q$ codes $X'$ and $Y'$ such that $X'\cap Y'=0$.
Let $Z$ be a complement of $X'+Y'$. 
The dimension of $Z$ is equal to $2k-n$.
The projective $[n,k]_q$ codes $X=X'+Z$ and $Y=Y'+Z$ satisfy \eqref{eq1-1}.
\end{proof}

\begin{rem}\label{rem-l14}{\rm
The latter statement holds for $q=2$ if $3\le k\le n-3$.
Consider the case when $n\ge 2k\ge 6$ and $q$ is an arbitrary.
Let $X$ and $Y$ be the projective $[n,k]_q$ codes represented by matrices 
\[M = [I, U-A, B], ~~~ N = [-I, I-U, B']\]
where $I$ is the identity matrix of rank $k$, $U$ is the rank $k$ matrix where each element is $1$ and the rank $k$ matrix $A=[a_{ij}]$ is as follows 
\[A ~ = ~ \left[\begin{array}{ccccccc}
0 & 1 & 0 &  ... & 0 & 0 \\
0 & 0 & 1 &  ... & 0 & 0\\
... & ... &  ... & ... & ... & ... \\
0 & 0 & 0 &  ... & 1 & 0 \\
0 & 0 & 0 &  ... & 0 & 1 \\
0 & 0 & 0 &  ... & 0 & 0
\end{array}\right]\]
(the only non-zero entries are $a_{ii+1}=1$).
The condition $k \geq 3$ guarantees that the matrices $[I, U-A]$ and $[-I, I-U]$ have no proportional columns.
As above, $B$ and $B'$ are $k\times (n-2k)$-matrices with no null column, no column proportional to a column of $[I, U-A]$ and $[-I, I-U]$ (respectively); they occur if $2k < n$. The rank of the matrix
\[\left[\begin{array}{c}
M\\
N
\end{array}\right] ~ = ~ \left[\begin{array}{ccc}
I & U-A & B \\
-I & I-U & B'
\end{array}\right]\]
 is equal to the rank of  the matrix
\[ \left[\begin{array}{ccc}
I & U-A & B \\  
O & I-A & B+B'
\end{array}\right].\]
Since 
\[\left[\begin{array}{cc}
I & U-A \\   
O & I-A 
\end{array}\right]\]
is an upper triangular matrix whose diagonal is formed by $1$, 
our matrices have rank $2k$ which means that $X\cap Y = 0$.
As in the proof of Lemma \ref{lemma1-4}, we show that the statement holds if $n-k<k\le n-3$.
}\end{rem}

\section{The graphs of simplex codes}
Through this section, we suppose that $n=[k]_q$. 
This is the case when  $\Pi(n,k)_{q}$ is the graph of $q$-ary simplex codes of dimension $k$.

\subsection{Simplex vectors and simplex codes}
If $X$ is a $q$-ary simplex code of dimension $k$,
then every hyperplane of $X$ is the intersection of $X$ with a certain coordinate hyperplane $C_{i}$
(indeed, $X$ contains precisely $[k]_q=n$ hyperplanes and all hyperplanes $X\cap C_{i}$ are mutually distinct).
Every $1$-dimensional subspace of $X$ is the intersection of precisely $[k-1]_q$ distinct $X\cap C_i$.
This means that every non-zero codeword $x\in X$ has precisely $[k-1]_q$ coordinates equal to $0$, i.e.
the Hamming weight of $x$ is $$[k]_q-[k-1]_q=q^{k-1}.$$
Then the Hamming distance between any two codewords of $X$ is equal to $q^{k-1}$
(indeed, every shifting $x\to x+y$ preserves the Hamming distance and it transfers  $X$ to itself if $x\in X$).

By \cite{B} (see also \cite[Theorem 7.9.5]{HP-book}), $q$-ary simplex codes of dimension $k$ can be characterized as subspaces of $V$
maximal with respect to the property that the Hamming weight of all non-zero vectors is equal to $q^{k-1}$.

We say that $x\in V$ is a {\it simplex vector} if the Hamming weight of this vector is $q^{k-1}$.
A non-zero vector is simplex if and only if it is a codeword of a certain $q$-ary simplex code of dimension $k$
(all codewords of simplex codes are simplex vectors and the group of monomial linear automorphisms of $V$
acts transitively on the set of simplex vectors and the set of simplex codes).
By the above characterization of simplex codes, a subspace of $V$ is a $q$-ary simplex code of dimension $k$
if and only if it is maximal with respect to the property that all non-zero vectors are simplex.

\begin{theorem}\label{theorem2}
The set of all simplex vectors is precisely the set of vectors $x=(x_{1},\dots, x_{n})$ satisfying all the equations
$$\sum_{i_1<\dots< i_{p^{j}}}x_{i_1}^{q-1}\dots\; x_{i_{p^{j}}}^{q-1}  =  0$$
for $j\in\{0,\dots,mk-m-1\}$, where $q=p^m$ and $p$ is a prime number.
\end{theorem}

\begin{proof}
We use the following Lucas's theorem:
for any prime number $p$ and expansions $a=\sum_{i=0}^r a_ip^i$ and $b=\sum_{i=0}^r b_ip^i$ we have
$${a\choose b}\equiv\prod_{i=0}^r {a_i\choose b_i} \;\; (\mathrm{mod}~ p);$$
in particular,
$${a\choose b}\not\equiv~ 0~ (\mathrm{mod}~ p)\;\Longleftrightarrow\; a_i\geq b_i\;\textrm{for all $i$.}$$
Every natural number can be written uniquely as $p^js$, where $j\geq 0$ and $s$ is not divisible by $p$. 
If the Hamming weight of $x=(x_{1},\dots,x_{n})$ is equal to such $p^js$, then 
the inequality $p^{j}s \le n=[k]_q=q^{k-1}+\dots +q+1$ implies that for this Hamming weight one of the following possibilities is realized:
\begin{enumerate}
\item[$\bullet$] $p^{j}s$ with $j\le mk-m-1$ and $s\ge 1$,
\item[$\bullet$] $p^{mk-m}=q^{k-1}$, i.e. $x$ is a simplex vector.
\end{enumerate}
For the first case Lucas's theorem shows that 
$$\sum_{i_1<\dots< i_{p^j}}x_{i_1}^{q-1}\dots\; x_{i_{p^j}}^{q-1} = {p^js\choose p^j}\ne 0.$$
In the second case, we obtain that 
$$\sum_{i_1<\dots< i_{p^j}}x_{i_1}^{q-1}\dots\; x_{i_{p^j}}^{q-1} = {p^{mk-k}\choose p^j}=0$$
for every $j\in\{0,\dots,mk-m-1\}$.
\end{proof}

\begin{cor}
All $q$-ary simplex codes of dimension $k$ can be characterized as subspaces of $V$ 
maximal with respect to the property that all vectors satisfy the equations from Theorem \ref{theorem2}.
\end{cor}

\begin{cor}
The graph $\Pi(7,3)_2$ consisting of all binary simplex codes of dimension $3$  
is isomorphic to the incidence graph $\Gamma_{1,3}({\mathbb F}^{4}_{2})$.
\end{cor}

\begin{proof}
Suppose that $q=2$ and $k=3$ (then $n=7$).
In this case,  the set of all simplex vectors $x=(x_{1},\dots,x_{7})$ is described by the following equations 
$$\sum^{7}_{i=1}x_{i}=0\;\mbox{ and }\;\sum_{i<j}x_{i}x_{j}=0.$$
Let $Q$ be the restriction of the quadratic form 
$\sum_{i<j}x_{i}x_{j}$ to the $6$-dimensional subspace $W$ defined by the first equation.
Then the set of all simplex vectors is the quadric of $Q$-singular vectors.
It is well-known that the restriction of the Grassmann graph $\Gamma_{3}(W)$ to the set of maximal $Q$-singular subspaces 
is isomorphic to $\Gamma_{1,3}({\mathbb F}^{4}_{2})$ (see, for example, \cite[Section 9.4]{drg-book}).
\end{proof}

\subsection{Two examples}
We construct two binary simplex codes $X,Y$ of dimension $4$ such that $\dim(X\cap Y)=2$ and 
there is no vertex of $\Pi(15,4)_2$ adjacent to both $X,Y$.
Consider the following $2$-dimensional subspaces of ${\mathbb F}^{15}_{2}$ formed by simplex vectors
$$L_1=\left\{
\begin{array}{c}
u_1\\
u_2\\
u_1+u_2\\
\end{array}
\right\}=\left\{\begin{array}{cccccccccccccccc}
0&0&0&0&0&0&0&1&1&1&1&1&1&1&1\\
0&0&0&1&1&1&1&0&0&0&0&1&1&1&1\\
0&0&0&1&1&1&1&1&1&1&1&0&0&0&0\\
\end{array}
\right\},$$
$$L_2=\left\{
\begin{array}{c}
v_1\\
v_2\\
v_1+v_2\\
\end{array}
\right\}=\left\{\begin{array}{cccccccccccccccc}
0&1&1&0&0&1&1&0&1&0&1&0&1&0&1\\
1&0&1&0&1&0&1&0&0&1&1&0&1&1&0\\
1&1&0&0&1&1&0&0&1&1&0&0&0&1&1\\
\end{array}
\right\},$$
$$L_3=\left\{
\begin{array}{c}
w_1\\
w_2\\
w_1+w_2\\
\end{array}
\right\}=\left\{\begin{array}{cccccccccccccccc}
0&1&1&0&0&0&0&1&0&1&1&1&0&1&1\\
1&0&1&1&0&1&1&0&0&0&0&1&0&1&1\\
1&1&0&1&0&1&1&1&0&1&1&0&0&0&0\\
\end{array}
\right\}.$$
Then $L_{1}+L_{2}$ and $L_{2}+L_{3}$ are the binary simplex codes of dimension $4$ whose generator matrices are
$$\left[\begin{array}{cccccccccccccccc}
0&0&0&0&0&0&0&1&1&1&1&1&1&1&1\\
0&0&0&1&1&1&1&0&0&0&0&1&1&1&1\\
0&1&1&0&0&1&1&0&1&0&1&0&1&0&1\\
1&0&1&0&1&0&1&0&0&1&1&0&1&1&0\\
\end{array}
\right]$$
and
$$\left[\begin{array}{cccccccccccccccc}
0&1&1&0&0&1&1&0&1&0&1&0&1&0&1\\
1&0&1&0&1&0&1&0&0&1&1&0&1&1&0\\
0&1&1&0&0&0&0&1&0&1&1&1&0&1&1\\
1&0&1&1&0&1&1&0&0&0&0&1&0&1&1\\
\end{array}
\right],$$
respectively.
The Hamming distance between any two codewords in a binary simplex code of dimension $4$ is equal to $8$,
but the Hamming distance between any simplex vectors $x\in L_{1}$ and $y\in L_{3}$ is not equal to $8$.
If $Z$ is a $4$-dimensional subspace of ${\mathbb F}^{15}_{2}$ adjacent to both $L_{1}+L_{2}$ and $L_{2}+L_{3}$,
then $Z$ has non-zero intersections with $L_{1}$ and $L_{3}$, 
in other words, $Z$ contains a pair of simplex vectors which are not at Hamming distance $8$.
This means that $Z$ is not a binary simplex code.

Now, we describe two ternary simplex codes $X,Y$ of dimension $3$ such that $\dim (X\cap Y)=1$ and there is no vertex of $\Pi(13,3)_3$
adjacent to both $X,Y$.
Let $X$ and $Y$ be the ternary simplex codes of dimension $3$ whose generator matrices are
$$\left[
\begin{array}{c}
w\\
v_1\\
v_2\\
\end{array}
\right]=\left[\begin{array}{ccccccccccccc}
0&0&0&0&1&1&1&1&1&1&1&1&1\\
0&1&1&1&0&0&0&1&2&1&1&2&2\\
1&0&1&2&0&1&2&0&0&1&2&1&2\\
\end{array}
\right]$$
and 
$$\left[
\begin{array}{c}
w\\
u_1\\
u_2\\
\end{array}
\right]=\left[\begin{array}{ccccccccccccc}
0&0&0&0&1&1&1&1&1&1&1&1&1\\
1&0&1&1&0&0&2&0&1&1&1&2&2\\
2&1&0&1&0&1&0&2&0&1&2&1&2\\
\end{array}
\right],$$
respectively.
Consider all $3$-dimensional subspaces adjacent to both $X,Y$
(there are precisely $16$ such subspaces). Their generator matrices are the following 
$$\left[
\begin{array}{c}
w\\
v_1\\
u_1\\
\end{array}
\right]=\left[\begin{array}{ccccccccccccc}
0&0&\mathbf{0}&\mathbf{0}&1&1&1&1&1&1&1&1&1\\
0&1&\mathbf{1}&\mathbf{1}&0&0&0&1&2&1&1&2&2\\
1&0&\mathbf{1}&\mathbf{1}&0&0&2&0&1&1&1&2&2\\
\end{array}
\right]$$
$$\left[
\begin{array}{c}
w\\
v_1\\
u_2\\
\end{array}
\right]=\left[\begin{array}{ccccccccccccc}
0&\mathbf{0}&0&\mathbf{0}&1&1&1&1&1&1&1&1&1\\
0&\mathbf{1}&1&\mathbf{1}&0&0&0&1&2&1&1&2&2\\
2&\mathbf{1}&0&\mathbf{1}&0&1&0&2&0&1&2&1&2\\
\end{array}
\right]$$
$$\left[
\begin{array}{c}
w\\
v_1\\
u_1+u_2\\
\end{array}
\right]=\left[\begin{array}{ccccccccccccc}
0&\mathbf{0}&\mathbf{0}&0&1&1&1&1&1&1&1&1&1\\
0&\mathbf{1}&\mathbf{1}&1&0&0&0&1&2&1&1&2&2\\
0&\mathbf{1}&\mathbf{1}&2&0&1&2&2&1&2&0&0&1\\
\end{array}
\right]$$
$$\left[
\begin{array}{c}
w\\
v_1\\
u_1+2u_2\\
\end{array}
\right]=\left[\begin{array}{ccccccccccccc}
0&0&0&0&1&\mathbf{1}&\mathbf{1}&1&1&1&1&1&1\\
0&1&1&1&0&\mathbf{0}&\mathbf{0}&1&2&1&1&2&2\\
2&2&1&0&0&\mathbf{2}&\mathbf{2}&1&1&0&2&1&0\\
\end{array}
\right]$$
$$\left[
\begin{array}{c}
w\\
v_2\\
u_1\\
\end{array}
\right]=\left[\begin{array}{ccccccccccccc}
\mathbf{0}&0&\mathbf{0}&0&1&1&1&1&1&1&1&1&1\\
\mathbf{1}&0&\mathbf{1}&2&0&1&2&0&0&1&2&1&2\\
\mathbf{1}&0&\mathbf{1}&1&0&0&2&0&1&1&1&2&2\\
\end{array}
\right]$$
$$\left[
\begin{array}{c}
w\\
v_2\\
u_2\\
\end{array}
\right]=\left[\begin{array}{ccccccccccccc}
\mathbf{0}&0&0&\mathbf{0}&1&1&1&1&1&1&1&1&1\\
\mathbf{1}&0&1&\mathbf{2}&0&1&2&0&0&1&2&1&2\\
\mathbf{2}&1&0&\mathbf{1}&0&1&0&2&0&1&2&1&2\\
\end{array}
\right]$$
$$\left[
\begin{array}{c}
w\\
v_2\\
u_1+u_2\\
\end{array}
\right]=\left[\begin{array}{ccccccccccccc}
0&0&\mathbf{0}&\mathbf{0}&1&1&1&1&1&1&1&1&1\\
1&0&\mathbf{1}&\mathbf{2}&0&1&2&0&0&1&2&1&2\\
0&1&\mathbf{1}&\mathbf{2}&0&1&2&2&1&2&0&0&1\\
\end{array}
\right]$$
$$\left[
\begin{array}{c}
w\\
v_2\\
u_1+2u_2\\
\end{array}
\right]=\left[\begin{array}{ccccccccccccc}
0&0&0&0&1&1&1&\mathbf{1}&\mathbf{1}&1&1&1&1\\
1&0&1&2&0&1&2&\mathbf{0}&\mathbf{0}&1&2&1&2\\
2&2&1&0&0&2&2&\mathbf{1}&\mathbf{1}&0&2&1&0\\
\end{array}
\right]$$
$$\left[
\begin{array}{c}
w\\
v_1+v_2\\
u_1\\
\end{array}
\right]=\left[\begin{array}{ccccccccccccc}
0&0&0&0&1&1&1&1&\mathbf{1}&\mathbf{1}&1&1&1\\
1&1&2&0&0&1&2&1&\mathbf{2}&\mathbf{2}&0&0&1\\
1&0&1&1&0&0&2&0&\mathbf{1}&\mathbf{1}&1&2&2\\
\end{array}
\right]$$
$$\left[
\begin{array}{c}
w\\
v_1+v_2\\
u_2\\
\end{array}
\right]=\left[\begin{array}{ccccccccccccc}
0&0&0&0&1&1&\mathbf{1}&1&\mathbf{1}&1&1&1&1\\
1&1&2&0&0&1&\mathbf{2}&1&\mathbf{2}&2&0&0&1\\
2&1&0&1&0&1&\mathbf{0}&2&\mathbf{0}&1&2&1&2\\
\end{array}
\right]$$
$$\left[
\begin{array}{c}
w\\
v_1+v_2\\
u_1+u_2\\
\end{array}
\right]=\left[\begin{array}{ccccccccccccc}
0&0&0&0&1&1&\mathbf{1}&1&1&\mathbf{1}&1&1&1\\
1&1&2&0&0&1&\mathbf{2}&1&2&\mathbf{2}&0&0&1\\
0&1&1&2&0&1&\mathbf{2}&2&1&\mathbf{2}&0&0&1\\
\end{array}
\right]$$
$$\left[
\begin{array}{c}
w\\
v_1+v_2\\
u_1+2u_2\\
\end{array}
\right]=\left[\begin{array}{ccccccccccccc}
\mathbf{0}&\mathbf{0}&0&0&1&1&1&1&1&1&1&1&1\\
\mathbf{1}&\mathbf{1}&2&0&0&1&2&1&2&2&0&0&1\\
\mathbf{2}&\mathbf{2}&1&0&0&2&2&1&1&0&2&1&0\\
\end{array}
\right]$$
$$\left[
\begin{array}{c}
w\\
v_1+2v_2\\
u_1\\
\end{array}
\right]=\left[\begin{array}{ccccccccccccc}
\mathbf{0}&0&0&\mathbf{0}&1&1&1&1&1&1&1&1&1\\
\mathbf{2}&1&0&\mathbf{2}&0&2&1&1&2&0&2&1&0\\
\mathbf{1}&0&1&\mathbf{1}&0&0&2&0&1&1&1&2&2\\
\end{array}
\right]$$
$$\left[
\begin{array}{c}
w\\
v_1+2v_2\\
u_2\\
\end{array}
\right]=\left[\begin{array}{ccccccccccccc}
\mathbf{0}&\mathbf{0}&0&0&1&1&1&1&1&1&1&1&1\\
\mathbf{2}&\mathbf{1}&0&2&0&2&1&1&2&0&2&1&0\\
\mathbf{2}&\mathbf{1}&0&1&0&1&0&2&0&1&2&1&2\\
\end{array}
\right]$$
$$\left[
\begin{array}{c}
w\\
v_1+2v_2\\
u_1+u_2\\
\end{array}
\right]=\left[\begin{array}{ccccccccccccc}
0&\mathbf{0}&0&\mathbf{0}&1&1&1&1&1&1&1&1&1\\
2&\mathbf{1}&0&\mathbf{2}&0&2&1&1&2&0&2&1&0\\
0&\mathbf{1}&1&\mathbf{2}&0&1&2&2&1&2&0&0&1\\
\end{array}
\right]$$
$$\left[
\begin{array}{c}
w\\
v_1+2v_2\\
u_1+2u_2\\
\end{array}
\right]=\left[\begin{array}{ccccccccccccc}
0&0&0&0&\mathbf{1}&1&1&1&1&\mathbf{1}&1&1&1\\
2&1&0&2&\mathbf{0}&2&1&1&2&\mathbf{0}&2&1&0\\
2&2&1&0&\mathbf{0}&2&2&1&1&\mathbf{0}&2&1&0\\
\end{array}
\right].$$
None of these codes is projective.

\end{document}